\documentclass[12pt,a4paper]{article}
\usepackage{graphicx}
\usepackage{latexsym} 
\usepackage{amssymb,amsmath,stmaryrd,amscd,bbm,marvosym}

\newcommand{\N}{\ensuremath{\mathbb{N}}}
\newcommand{\Z}{\ensuremath{\mathbb{Z}}}

\newtheorem{theorem}{Theorem}[section]
\newtheorem{lemma}[theorem]{Lemma}
\newtheorem{conjecture}[theorem]{Conjecture}

\newcounter{claim}

{\refstepcounter{claim}\vspace{1ex}\noindent{(\it\arabic{claim}){#1}{}}\it}{\vspace{1ex}}

	{\noindent {}{#1}{}}{ This proves~(\arabic{claim}).\vspace{1ex}}

 \newenvironment{proof}[1][]%
 {\noindent {\setcounter{claim}{0}\sc proof ---
    }{#1}{}}{\hfill$\Box$\vspace{2ex}}

\usepackage{ifpdf}

\ifpdf
\fi

\title{On the structure of self-complementary graphs}

\author{Nicolas  Trotignon\thanks{Laboratoire  Leibniz,  IMAG,  46  av
  F\'elix      Viallet,      38041 cedex,      Grenoble,      France.\newline
  \protect\rule{1.8em}{0cm}nicolas.trotignon@imag.fr} }

\date{April 6 2005}
\begin{document}

\maketitle
{\small
\noindent{\bf Abstract:} A  \emph{self-complementary} graph is a graph
isomorphic  to its  complement.  An  isomorphism between  $G$  and its
complement,  viewed as  a permutation  of  $V(G)$, is  then called  an
\emph{antimorphism}.  A \emph{skew partition} of $G$ is a partition of
$V(G)$ into 4~sets $A,B,C,D$ such  that there is no edge between $A,B$
and every possible edge between $C,D$. A \emph{symmetric partition} of
$G$ is a partition of $V(G)$  into 4~sets $A,B,C,D$ such that there is
no edge  between $A, D$, no  edge between $B, C$,  every possible edge
between $A,B$ and every possible edge between $C,D$.

We  give  a  new  proof  of  a theorem  of  Gibbs  saying  that  every
self-complementary graph  on $4k$ vertices  has $k$ disjoint  paths on
$4$  vertices  as  induced   subgraph.   This  new  proof  gives  more
structural  information than  the  original one.   We conjecture  that
every self-complementary graph on  $4k$ vertices either has an induced
cycle on  5~vertices, or a  skew partition, or a  symmetric partition.
The new proof  of Gibb's theorem yields a proof  of the conjecture for
the self-complementary  graphs that have  an antimorphism that  is the
product of a two circular permutations, one of them of length~4.  }

\section{Introduction}

In  this paper  graphs  are  simple, non-oriented,  with  no loop  and
finite.  Several definitions that can  be found in most handbooks (for
instance~\cite{diestel:graph}  for graphs  and \cite{garey.johnson:np}
for algorithms) will not be given.  We also refer the reader to a very
complete    survey    on     self-complementary    graphs    due    to
Farrugia~\cite{farrugia:these}.

If $G$ is a graph, we  denote the complement of $G$ by $\overline{G}$.
A graph is  said to be \emph{self-complementary} if  $G$ is isomorphic
to its complement $\overline{G}$. We will often write ``sc-graph'' for
``self-complementary graph''.  It is very  easy to construct a  lot of
examples of  sc-graphs: take any  graph $G$, and  consider 2~copies of
$G$  say  $G_1,  G_2$,   and  2~copies  of  $\overline{G}$  say  $G_3,
G_4$. Then join every vertex of  $G_1$ to every vertex of $G_3$, every
vertex of $G_3$ to every vertex  of $G_4$ and every vertex of $G_4$ to
every vertex of  $G_2$. The graph obtained is  self-complementary, and
we    call    it   the    graph    obtained    from    $G$   by    the
\emph{$P_4$-construction}.
 
\begin{figure}[bh]
  \center
  \includegraphics{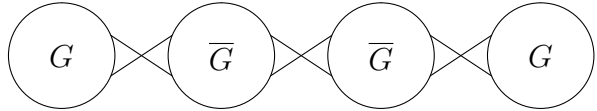}
  \caption{The $P_4$-construction applied to $G$}
\end {figure}

An   important   problem   in   algorithmic  graph   theory   is   the
\emph{isomorphism problem},  known to  be difficult and  unsettled: is
there a polynomial time algorithm  that decides whether two graphs are
isomorphic~? If there is one,  then we can easily decide in polynomial
time if a  graph is self-complementary, by just  running the algorithm
on          $G,          \overline{G}$.           Colbourn          et
al.~\cite{colbourn.colbourn:autoiso}  studied the converse  and proved
that  the  recognition  of sc-graphs  is  \emph{isomorphism-complete}.
That is: if there exists a polynomial time algorithm that decides if a
graph  is  self-complementary, then  there  exists  a polynomial  time
algorithm for the isomorphism problem.   The result of Colbourn et al.
is  not  so surprising  because  of  the $P_4$-construction  described
above.  Consider two graphs $G, H$. Consider $G_1, G_2$, two copies of
$G$ and $H_1,  H_2$ two copies of $\overline{H}$.   Construct (like in
the $P_4$-construction) a  new graph $J$ : join  every vertex of $G_1$
to every  vertex of $H_1$,  every vertex of  $H_1$ to every  vertex of
$H_2$, and every vertex of $H_2$  to every vertex of $G_2$.  To decide
if $J$ is  self-complementary, the obvious way is to  decide if $G, H$
are isomorphic.  The result of Colbourn  et al.  says that there is no
better  way in  general.  So,  it  is to  be feared  that despite  (or
because of) formal equivalence, a study of the properties of sc-graphs
will not help in solving the isomorphism problem.

However, the structure of sc-graphs is worth investigating for its own
interest  and because particular  sc-graphs have  sometime interesting
properties,   as   being    smallest   counter-examples   to   several
conjectures\footnote{$C_5$   is   the   smallest   non-perfect   graph
see~\cite{livre:perfectgraphs},   $L(K_{3,3}\setminus   e)$   is   the
smallest  perfect graph  with no  even pair  and no  even pair  in its
complement,   see~\cite{everett.f.l.m.p.r:ep}.     There   are   other
examples.}.  It could  also help for a general  construction for every
sc-graph, or at least for  some substantial subclasses.  Note that the
$P_4$-construction is not  a good candidate: in a graph  on at least 8
vertices obtained  by the $P_4$-construction, every  vertex has degree
at least 2, and  on figure~\ref{fig8} page~\pageref{fig8}, there is an
sc-graph  with   a  vertex  of  degree   one.   Moreover,  recognition
algorithms for  special classes of sc-graphs can  be drasticaly easier
than the isomorphism problem for  the same class.  For instance, it is
easy to  see that  the only triangle-free  sc-graphs are  the isolated
vertex, $P_4$ and $C_5$, because by the Ramsey's famous Theorem, every
graph on  at least 6~vertices  has a triangle  or the complement  of a
triangle.   Thus, recognizing  triangle-free sc-graphs  is  trivial in
constant time  while the isomorphism problem  for triangle-free graphs
is difficult.  It  might be possible to recognize  special non-trivial
classes of  sc-graphs in polynomial  time.  After reading  this paper,
the  reader will maybe  want to  look for  a general  construction for
$C_5$-free sc-graphs, and  why not for a recognition  algorithm (he or
she must be warned that most of the work is still to be done\dots).

In this  paper, we aim  at structural properties of  sc-graphs, saying
something like: every sc-graph either contains some prescribed induced
subgraph  or can  be  partitioned  into sets  of  vertices with  some
prescribed adjacencies.   There are really  few such results.   In his
master's  thesis  that surveys  more  than  400  papers on  sc-graphs,
Farrugia~\cite{farrugia:these} mentions only one theorem due to Gibbs:

\begin{theorem}[Gibbs, \cite{gibbs:sc}]
  \label{th:gibbs}
  An sc-graph on $4k$ vertices contains $k$ disjoint induced $P_4$'s.
\end{theorem}

As pointed out  by Farrugia, the theorem above  has two major defaults
in view  of algorithmic applications.  First, the  problem of deciding
whether  the  vertices of  a  graph can  be  partitioned  into sets  of
4~vertices, each of  them inducing a $P_4$, is  NP-complete (proved by
Kirkpatrick and  Hell,~\cite{kirkpatrick.hell:gm}).  Secondly, even if
the  partition  into  $P_4$'s  of  an  sc-graph  is  obtained  by  any
unexcepected mean, it will be  of no use for recursion, since removing
blindly one or some  of the $P_4$'s may yield a graph  that is no more
self-complementary  and  that  will  have  in  general  no  forseeable
properties.

We will  investigate structural properties  of sc-graphs that  fix the
first default:  the structures  that we will  find (or  conjecture) in
sc-graphs will  be detectable in polynomial  time.  Unfortunately, our
results (and conjectures) will still  have the second default: we will
be able to break several  sc-graphs into pieces with special adjacency
properties, but without garanteeing  any hereditary properties on these
pieces.

We will first give a new proof  of the theorem of Gibbs, that yields a
slightly  different  result  and  gives  more  structural  information
(Section~\ref{sec:gibbs}).  This will allow us to prove a special case
of a  conjecture: every  sc-graph on $4k$  vertices either  contains a
$C_5$ as an induced subgraph or can be broken in 4 pieces with special
adjacencies properties (Section~\ref{sec:conj}).  Page~\pageref{fig8},
we show a picture of all the sc-graphs on 8 vertices.

\begin{figure}[h]
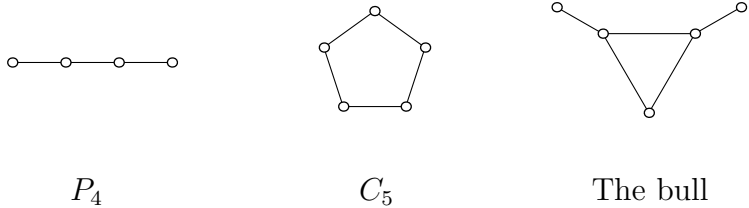

  \center
  \begin{tabular}{ccccc}
  \parbox[c]{0cm}{\rule{0cm}{3cm}}
  \includegraphics{fig.autocomp.11}&\rule{1cm}{0cm}&
  \parbox[c]{2cm}{\includegraphics{fig.autocomp.12}}&\rule{.2cm}{0cm}&
  \parbox[c]{3.3cm}{\includegraphics{fig.autocomp.13}}\\
   $P_4$&& \hspace{-.6cm} $C_5$&& \hspace{-.65cm}The bull
  \end{tabular}
\caption{The 3 sc-graphs on 4 or 5 vertices\label{base.fig.45}}
\end{figure}

\section{A new proof of Gibb's theorem}
\label{sec:gibbs}

If $G$  is a  graph, we  denote by $V(G)$  the vertex  set of  $G$, by
$E(G)$ the edge  set of $G$.  If $A\subset V(G)$,  we denote by $G[A]$
the subgraph  of $G$ induced  by~$A$.  If $v$  is a vertex of  $G$, we
denote by  $N(v)$ the  set of the  neighbours of  $v$ .  We  denote by
$\overline{N}(v)$  the set of  the non-neighbours  of $v$.   Note that
$v\in \overline{N}(v)$.  If $uv \in E(G)$, we say that $u$ \emph{sees}
$v$, and if $uv \notin E(G)$, we say that $u$ \emph{misses} $v$.

By the  definition, a graph $G$  is self-complementary if  and only if
there exists  a bijection $\tau$ from  $V(G)$ to $V(G)$  such that for
every pair $\{a,b\}$  of distinct vertices we have:  $\{a,b\} \in E(G)
\Leftrightarrow \{\tau(a),  \tau(b) \} \notin E(G)$.   Such a function
$\tau$ is called an \emph{antimorphism} of $G$.

Sachs~\cite{sachs:sc}  and  Ringel~\cite{ringel:sc}  proved  that  any
antimorphism is  a product of circular permutations  whose lengths are
all multiples  of~4, except possibly  for one of length~1.   Note that
this implies a well known fact:  the number of vertices of an sc-graph
is equal  to 0 or  1 modulo~4.  Gibbs~\cite{gibbs:sc} also  proved the
following:

\begin{theorem}[Gibbs~\cite{gibbs:sc}]
  \label{th:po2}
  If $G$ is  an sc-graph, then there exists  an antimorphism $\tau$ of
  $G$  such that  every circular  permutation of  $\tau$ has  length a
  power of~2.
\end{theorem}

It  is convenient  to denote  by $(a_1  a_2 \dots  a_k)$  the circular
permutation of $\{a_1  a_2 \dots a_k\}$ that maps  $a_i$ to $a_{i+1}$,
where  the addition of  the subscripts  is taken  modulo $k$.   When a
circular permutation has length $4k$,  we often denote it by $(a_1 b_1
c_1 d_1  a_2 b_2  c_2 d_2  \dots a_k b_k  c_k d_k)$.   Implicitly, the
subscripts are  then taken modulo  $k$ (for instance $a_{k+3}  = a_3$,
$d_0 = d_k$,~\dots).

We recall here a lemma  used by Gibbs to prove Theorem~\ref{th:gibbs}.
We give his proof with our notation.
 
\begin{lemma}[Gibbs~\cite{gibbs:sc}]
  \label{l:gibbs}
  Let $k\geq  1$ be  an integer  and let $G$  be an  sc-graph  with an
  antimorphism $\tau = (a_1 b_1 c_1  d_1 a_2 b_2 c_2 d_2 \dots a_k b_k
  c_k d_k)(\dots) \cdots (\dots)$. Then either:

  \begin{itemize}
  \item 
    There exists $i\in \N$ such  that $\{a_1, b_1, a_i, b_i\}$ induces
    a $P_4$ for which $(a_1 b_1 a_i b_i)$ is an antimorphism.

  \item 
    There exists $i \in \N$  such that $\{a_1, b_1, c_i, d_i\}$ induces
    a $P_4$ for which $(a_1 b_1 c_i d_i)$ is an antimorphism.
  \end{itemize}
\end{lemma}

\begin{proof}
  Let us  suppose without loss  of generality that $a_1$  misses $b_1$
  (if not, we may replay  the same proof in $\overline{G}$).  Applying
  $\tau^{-1}$,  we know  that $a_1$  sees  $d_k$.  So  there exists  a
  smallest integer  $i>1$ such  that: $a_1$ sees  $b_i$ or  $a_1$ sees
  $d_i$.

  If $a_1$  sees $b_i$ then,  $i\geq 2$.  Applying  $\tau^{4(i-1)}$ to
  $a_1$ and $b_1$, we know that $a_i$ misses $b_i$.  By the definition
  of $i$, we know that  $a_1$ misses $d_{i-1}$. Thus, applying $\tau$,
  $b_1$ sees $a_i$.  If $a_1$ sees $a_i$, then, applying $\tau$, $b_1$
  misses $b_i$  and $\{a_1, b_1,  a_i, b_i\}$ induces $P_4$  for which
  $(a_1 b_1  a_i b_i)$ is an  antimorphism. In the same  way, if $a_1$
  misses  $a_i$,  then  $b_1$  sees   $b_i$  and  we  reach  the  same
  conclusion.

  If $a_1$  misses $b_i$,  then by the  definition of $i$,  $a_1$ sees
  $d_i$.  Applying $\tau^{4(i-1)+2}$ to  $a_1$ and $b_1$, we know that
  $c_i$ misses  $d_i$.  Applying  $\tau$ to $a_1$  and $b_i$,  we know
  that $b_1$ sees $c_i$.  If  $a_1$ sees $c_i$, applying $\tau$, $b_1$
  misses $d_i$ and $\{a_1, b_1, c_i, d_i\}$ induces a $P_4$ for which
  $(a_1 b_1 c_i  d_i)$ is an antimorphism.  By the  same way, if $a_1$
  misses $c_i$ then $b_1$ sees $d_i$ and we reach the same conclusion.
\end{proof}

We propose a new lemma of  the same flavour that gives more structural
information  on sc-graphs.   To state  it,  we need  a definition.   A
\emph{symmetric partition}  in a graph $G$ is  a partition $(A,B,C,D)$
of $V(G)$ such that each of $A,B,C,D$ is non-empty, there are no edges
between $A,D$,  no edges between  $B,C$, every possible  edges between
$A,B$, and every possible edges between $C, D$.

\begin{lemma}
  \label{l:base} 
  Let  $k\geq  1$  be an  integer  and  $G$  be  an sc-graph  with  an
  antimorphism $\tau = (a_1 b_1 c_1  d_1 a_2 b_2 c_2 d_2 \dots a_k b_k
  c_k d_k)(\dots)\cdots (\dots)$. 

  \noindent Put $A = \{a_1,  \dots, a_k\}$, $B = \{b_1, \dots, b_k\}$,
  $C = \{c_1, \dots, c_k\}$, $D = \{d_1, \dots, d_k\}$.
  Then either:

  \begin{itemize}
  \item 
    There  exists  $i,j  \in  \N$  such  that  $\{a_1,  b_i,  a_{1+j},
    b_{i+j}\}$ induces a $P_4$  for which $(a_1 b_i a_{1+j} b_{i+j})$
    is an antimorphism.

  \item 
    $(A, B, C,  D)$ is a symmetric partition of $G[A\cup  B \cup C \cup
    D]$.
    
  \item 
    $(B, C, D,  A)$ is a symmetric partition of $G[A\cup  B \cup C \cup
    D]$.
  \end{itemize}
\end{lemma}

\begin{proof}
  If  every vertex  in $A$  sees every  vertex in  $B$,  then applying
  $\tau$  three times,  we  see that  $(A,  B, C,  D)$  is a  symmetric
  partition of $G[A\cup B \cup C \cup D]$. Similarly, if every vertex
  in $A$ misses every vertex in $B$  then $(B, C, D, A)$ is a symmetric
  partition of $G[A\cup  B \cup C \cup D]$.  Thus,  we may assume that
  some vertex $a_h$ in $A$  has neighbours and non-neighbours in $B$,
  and applying  $\tau^{4(h-1)}$, we see that $a_1$  has neighbours and
  non-neighbours in $B$.

  Suppose  first that  $a_1$  has  at least  as  many neighbours  than
  non-neiboughs  in  $B$,  more   precisely:  $|N(a_1)  \cap  B|  \geq
  |\overline{N}(a_1) \cap  B|$. Let $i$  be such that $a_1  b_i \notin
  E(G)$.   There exists  $j  \not \equiv  0  \pmod k$  such that  $a_1
  b_{i-j}  \in   E$  and   $  a_1  b_{i+j}   \in  E$,   for  otherwise
  $|\overline{N}(a_1)  \cap  B|  >  k/2  \geq |N(a_1)  \cap  B|$  ,  a
  contradiction. Note  that we  may have $b_{i-j}  = b_{i+j}$  if $i-j
  \equiv i+j \pmod k$

  We  already  know  $a_1  b_i  \notin E$.   Applying  $\tau^{4j}$  to
  $a_1b_i$ we know  $a_{1+j} b_{i+j} \notin E$.  We  already know $a_1
  b_{i-j}  \notin E$.   Applying $\tau^{4j}$  to $a_1b_{i-j}$  we know
  $a_{1+j}  b_{i}  \in  E$.  If  $a_1$ sees  $a_{1+j}$  then  applying
  $\tau^{1+4(i-1)}$, $b_i$ misses  $b_{i+j}$ and $\{a_1, b_i, a_{1+j},
  b_{i+j}\}$ induces a $P_4$ for  which $(a_1 b_i a_{1+j} b_{i+j})$ is
  an   antimorphism.   If   $a_1$  misses   $a_{1+j}$   then  applying
  $\tau^{1+4(i-1)}$,  $b_i$  sees  $b_{i+j}$  and we  reach  the  same
  conclusion.

  We   are  left   with  the   case  where   $|N(a_1)  \cap   B|  \leq
  |\overline{N}(a_1) \cap B|$.  But then, the proof is similar up to a
  complementation of $G$.
\end{proof}

Note that if $(A,B,C,D)$ is  a symmetric partition then for any $i,j,l
\in \N$, the  set $\{a_i, b_{i+j}, c_{l}, d_{l+j}\}$  induces a $P_4$.
Because by  the definition of  symmetric partitions, we have  $a_i b_i
\in E$, $b_i c_i  \notin E$, $c_i d_i \in E$, $d_i  a_i \notin E$, and
applying $\tau$, exactely one of $a_i  c_i$, $b_i d_i$ is an edge.  If
$(B,C,D,A)$  is a symmetric  partition we  reach the  same conclusion.
This remark allows us to follow the lines of Gibbs, and to prove again
his theorem  (Theorem~\ref{th:gibbs}) using Lemma~\ref{l:base} instead
of Lemma~\ref{l:gibbs}. Let us do it for the sake of completeness.

Consider an sc-graph on $4k$  vertices and an antimorphism $\tau$.  By
theorem~\ref{th:po2}  we may  assume that  every cycle  of  $\tau$ has
length a power of~2. Let  us consider a circular permutation $(a_1 b_1
c_1 d_1 a_2  b_2 c_2 d_2 \dots a_k  b_k c_k d_k)$ of $\tau$.  Put $A =
\{a_1, \dots a_k\}$, $B = \{b_1,  b_2, \dots, b_k\}$, $C = \{c_1, c_2,
\dots, c_k\}$, $D  = \{d_1, d_2, \dots, d_k\}$.  We  claim that we may
partition $A\cup  B\cup C \cup  D$ in sets  of 4~vertices all  of them
inducing a $P_4$, thus proving the theorem.

We have  $\tau = (a_1 b_1  c_1 d_1 a_2 b_2  c_2 d_2 \dots  a_k b_k c_k
d_k)(\dots)\cdots  (\dots)$.  Apply  Lemma~\ref{l:base}.   If  one  of
$(A,B,C,D)$,  $(B,C,D,A)$ is  a  symmetric partition,  then for  every
$i\in \N$, $\{a_i, b_i, c_i, d_i\}$  induces a $P_4$ and we may easily
partition $A\cup B\cup  C \cup D$ into sets of  4~vertices all of them
inducing a $P_4$. So we are  left with the case where there exists $i$
such that $\{a_1, b_i, a_{1+j},  b_{i+j}\}$ induces a $P_4$ for which
$(a_1  b_i a_{1+j}  b_{i+j})$  is  an antimorphism.   Let  us put  $k=
2^{\alpha}$, $\alpha \geq 1$.  Note that for any $l$, $\{a_l, b_{l+i},
a_{l+j}, b_{l+i+j} \}$ induces a $P_4$ (apply $\tau^{4(l-1)}$) that we
denote by $P^l$.

We  claim that  we can  choose $l_1,  l_2, \dots,  l_{k/2}$  such that
$(P^{l_1},  P^{l_2}, \dots, P^{l_{k/2}})$  partitions $A\cup  B$.  For
that purpose it suffices to prove that for any integer $j \not\equiv 0
\pmod {2^\alpha}$,  $\Z_{2^\alpha}$ can  be partitioned into  pairs of
the form $\{l, l+j\}$.   For $\alpha = 1$ this can be  done, so let us
prove it by  induction.  If $j$ is even, then let  us partition by the
induction hypothesis $\Z_{2^{\alpha-1}}$ into  pairs of the form $\{l,
l+j/2\}$: $\Z_{2^{\alpha-1}}  = \cup_i \{l_i,  l_i + j/2\}$.   Then we
have  $\Z_{2^{\alpha}} =  \cup_i [\{2l_i,  2l_i +  j\}  \cup \{2l_i+1,
2l_i+1  +  j\}]$,  so  we  manage to  partition  $\Z_{2^{\alpha}}$  as
desired.   If  $j$ is  odd,  then starting  from  $1$  and adding  $j$
successively, we build a Hamiltonian cycle going through every element
of its vertex  set $\Z_{2^{\alpha}}$.  Take then every  second edge of
this   Hamiltonian  cycle:   here   are  the   pairs  that   partition
$\Z_{2^{\alpha}}$ as  desired.  So $A\cup  B$ may be  partitioned into
$P_4$'s.   Applying  $\tau^2$, we  see  that  $C\cup  D$ can  also  be
partitioned  into  $P_4$'s,  hence  $A\cup  B\cup C  \cup  D$  can  be
partitioned into $P_4$'s.

\section{A theorem and a conjecture}
\label{sec:conj}

A \emph{skew partition}  in a graph $G$ is  a partition $(A,B,C,D)$ of
$V(G)$ such  that each of $A,B,C,D$  is non-empty, there  are no edges
between $A,B$ and every possible edges between $C,D$. We are now able
to prove the following:
  
\begin{theorem}
  \label{th:m}
  Let  $G$ be  an sc-graph  with an  antimorphism $\tau$  that  is the
  product of two circular permutations,  one of them of length~4. Then
  either:

  \begin{itemize}
  \item
    $G$ contains a $C_5$ as an induced subgraph;
  \item
    $G$ contains a skew partition;
  \item 
    $G$ contains a symmetric partition.
  \end{itemize}

\end{theorem}

\begin{proof}
  Let us call $a,b,c,d$ the four vertices of a cycle of $\tau$. If the
  other cycle  has length~1, then $G$  is the $C_5$ or  the bull which
  has  a skew  partition  (see figure~\ref{base.fig.45}).   So we  may
  assume $\tau = (abcd)(a_1 b_1 c_1  d_1 a_2 b_2 c_2 d_2 \dots a_k b_k
  c_k  d_k)$ with  $k\geq 1$.   Put $A  = \{a_1,  \dots, a_k\}$,  $B =
  \{b_1, \dots,  b_k\}$, $C =  \{c_1, \dots, c_k\}$,  $D=\{d_1, \dots,
  d_k\}$.  Let  us suppose up  to a circular permutation  of $a,b,c,d$
  that the three edges of  $G[a,b,c,d]$ are $ab$, $ac$ and $cd$.  Note
  that if  $a$ sees $a_1$,  then $a$ sees  every vertex in  $A$ (apply
  $\tau^{4i}, i\in \N$).  By the  same way, if $v$ is in $\{a,b,c,d\}$
  and if $H$ is one of $A,B,C,D$, then either $v$ sees every vertex in
  $H$,  or  $v$  misses  every  vertex  in  $H$.   For  every  $v  \in
  \{a,b,c,d\}$, put $N_v  = N(v)\cap(A \cup B\cup C\cup  D)$.  We deal
  now  with the $2^4=16$  following cases,  according to  $N_a$.  Note
  that once $N_a$ is known, $N_b$,  $N_c$ and $N_d$ are also known, by
  applying $\tau$  three times.   For some of  the 16~cases,  we apply
  Lemma~\ref{l:base} to  the cycle $(a_1 b_1  c_1 d_1 a_2  b_2 c_2 d_2
  \dots  a_k  b_k  c_k  d_k)$  of  $\tau$.  Then,  up  to  a  circular
  permutation of $A,B,C,D$, we  may suppose that either $(A,B,C,D)$ is
  a symmetric partition of $G[A \cup B \cup C \cup D]$, or there exist
  $i,j \in  \N$ such that  $\{a_1, b_i, a_{1+j}, b_{i+j}\}$  induces a
  $P_4$ with $a_1 a_{1+j}$ as central edge.
   
  \newpage
  \begin{enumerate}

  \item $N_a= A \cup B$.

    Then $N_b = A  \cup D$, $N_c = C \cup D$ and $N_d  = B \cup C$. If
    $\{a_1,  b_i,  a_{1+j},  b_{i+j}\}$  induces  a  $P_4$  with  $a_1
    a_{1+j}$ as central edge, then $\{b, a_1, b_i, a_{1+j}, b_{i+j}\}$
    induces  a $C_5$. Else,  $(A,B,C,D)$ is  a symmetric  partition of
    $G[A \cup B  \cup C \cup D]$.  We see that $(A \cup  \{a\}, B \cup
    \{b\}, C\cup  \{c\}, D  \cup \{d\})$ is  a symmetric  partition of
    $G$.

  \item $N_a= C \cup D$.

    Then $N_b = B  \cup C$, $N_c = A \cup B$ and $N_d  = A \cup D$. If
    $\{a_1,  b_i,  a_{1+j},  b_{i+j}\}$  induces  a  $P_4$  with  $a_1
    a_{1+j}$ as central edge, then $\{b, a_1, b_i, a_{1+j}, b_{i+j}\}$
    induces a $C_5$. Else, $(A,B,C,D)$ is a symmetric partition of $G[A
    \cup B \cup C \cup D]$. We  see that $(A \cup \{c\}, B \cup \{d\},
    C\cup \{a\}, D \cup \{b\})$ is a symmetric partition of $G$.

  \item $N_a= A \cup C$.

    Then  $N_b = N_c  = N_d  = A  \cup C$.   Thus $(\{b,d\},  B\cup D,
    \{a,c\}, A\cup C)$ is a skew partition.

  \item $N_a= B \cup D$.

    Then  $N_b = N_c  = N_d  = B  \cup D$.   Thus $(\{a,c\},  A\cup C,
    \{b,d\}, B\cup D)$ is a skew partition.

  \item $N_a= A \cup D$.

    Then $N_b = C  \cup D$, $N_c = B \cup C$ and $N_d  = A \cup B$. If
    $\{a_1,  b_i,  a_{1+j},  b_{i+j}\}$  induces  a  $P_4$  with  $a_1
    a_{1+j}$ as central edge, then $\{c, a_1, b_i, a_{1+j}, b_{i+j}\}$
    induces a $C_5$. Else, $(A,B,C,D)$ is a symmetric partition of $G[A
    \cup B \cup C \cup D]$. We  see that $(A \cup \{c\}, B \cup \{d\},
    C\cup \{a\}, D \cup \{b\})$ is a symmetric partition of $G$.

  \item $N_a= B \cup C$.

    Then $N_b = A  \cup B$, $N_c = A \cup D$ and $N_d  = C \cup D$. If
    $\{a_1,  b_i,  a_{1+j},  b_{i+j}\}$  induces  a  $P_4$  with  $a_1
    a_{1+j}$ as central edge, then $\{a, a_1, b_i, a_{1+j}, b_{i+j}\}$
    induces a $C_5$. Else, $(A,B,C,D)$ is a symmetric partition of $G[A
    \cup B \cup C \cup D]$. We  see that $(A \cup \{a\}, B \cup \{b\},
    C\cup \{c\}, D \cup \{d\})$ is a symmetric partition of $G$.

  \item $N_a = \emptyset$.
    
    Then $a_1$ sees $b$, misses $c$  and sees $d$. Thus, $\{a_1, a, b,
    c, d\}$ induces $C_5$.

  \item $N_a = A \cup B \cup C \cup D$.
    \label{case:all}
    Then $N_c  = N_a = A\cup B  \cup C \cup D$.   Thus $(\{b\}, \{d\},
    \{a,c\}, A\cup B \cup C \cup D)$ is a skew partition.

  \item $N_a = A$.

    Then $N_b = A  \cup C \cup D$, $N_c = C$ and $N_d  = A \cup B \cup
    C$.   Thus  $(\{a,c\}, B\cup  D,  \{b,d\},  A\cup  C)$ is  a  skew
    partition.
 
  \item $N_a= B \cup C \cup D$.

    Then $N_b  = B$,  $N_c = A  \cup B  \cup D$ and  $N_d =  D$.  Thus
    $(\{b,d\}, A\cup C, \{a,c\}, B\cup D)$ is a skew partition.
   
  \item $N_a= B$.

    Then $N_b = A  \cup B \cup D$, $N_c = D$ and $N_d  = B \cup C \cup
    D$.   Thus  $(\{a,c\}, A\cup  C,  \{b,d\},  B\cup  D)$ is  a  skew
    partition.

  \item $N_a= A \cup C \cup D$.

    Then $N_b  = C$,  $N_c = A  \cup B  \cup C$ and  $N_d =  A$.  Thus
    $(\{b,d\}, B\cup D, \{a,c\}, A\cup C)$ is a skew partition.

  \item $N_a= C$.

    Then $N_b = A  \cup B \cup C$, $N_c = A$ and $N_d  = A \cup C \cup
    D$.   Thus  $(\{a,c\}, B\cup  D,  \{b,d\},  A\cup  C)$ is  a  skew
    partition.

  \item $N_a= A \cup B \cup D$.

    Then $N_b  = D$,  $N_c = B  \cup C  \cup D$ and  $N_d =  B$.  Thus
    $(\{b,d\}, A\cup C, \{a,c\}, B\cup D)$ is a skew partition.

  \item $N_a= D$.

    Then $N_b = B  \cup C \cup D$, $N_c = B$ and $N_d  = A \cup B \cup
    D$.   Thus  $(\{a,c\}, A\cup  C,  \{b,d\},  B\cup  D)$ is  a  skew
    partition.

  \item $N_a= A \cup B \cup C$.

    Then $N_b  = A$,  $N_c = A  \cup C  \cup D$ and  $N_d =  C$.  Thus
    $(\{b,d\}, B\cup D, \{a,c\}, A\cup C)$ is a skew partition.

  \end{enumerate}
\end{proof}

Note   that  as  pointed   out  by   Farrugia,  a   generalisation  of
Case~\ref{case:all} of  the proof was implicitly known  by Akiyama and
Harary~\cite{akiyama.harary:81}.  They proved  that if an sc-graph $G$
has at least an end-vertex,  then $G$ has exactly two end-vertices $b,
d$  and exactly  two cut  vertices $a,c$.   They proved  that $(\{b\},
\{d\},  \{a,c\}, V(G)\setminus\{a,b,c,d\})$ is  then a  skew partition
of~$G$.
 
Let  us  now  discuss  the  motivivation and  possible  extensions  of
Theorem~\ref{th:m}.  Skew partitions  were introduced by Chv\'atal for
the  study of  perfect graphs~\cite{chvatal:starcutset},  and  play an
important  role in  the proof  of Strong  Perfect Graph  Conjecture by
Chudnovsky,              Robertson             Seymour             and
Thomas~\cite{chudvovsky.r.s.t:spgt}. Symmetric  partitions may be seen
as  a very particular  case of  the 2-join  defined by  Cunningham and
Cornu\'ejols,    once    again    for    the    study    of    perfect
graphs~\cite{cornuejols.cunningham:2join}. A~\emph{2-join} in $G$ is a
partition  $(X_1,  X_2)$  of  $V(G)$  so  that  there  exist  disjoint
non-empty $A_i, B_i \subset X_i$, ($i=1, 2$) satisfaying:

\begin{enumerate} 
\item
  every vertex of $A_1$ is adjacent to every vertex of $A_2$ and every
  vertex of $B_1$ is adjacent to every vertex of $B_2$;
\item
  there are no other edges between $X_1$ and $X_2$;
\item 
  for  $i= 1, 2$,  every component  of $G[X_i]$  meets both  $A_i$ and
  $B_i$;
\item
  for $i= 1, 2$,  if $|A_i| = |B_i| = 1$, and  if $X_i$ induces a path
  of $G$ joining the vertex of  $A_i$ and the vertex of $B_i$, then it
  has length at least~3.
\end{enumerate}

The  conditions 3,  4  are called  the \emph{technical  requirements}.
They  are important for  algorithms, and  for applications  to perfect
graphs.  If  a graph $G$  has a 2-join  with the above  notation, then
$(A_1, A_2,  B_1, B_2)$  is a symmetric  partition of $G[A_1  \cup A_2
\cup  B_1 \cup  B_2]$.  In  other words,  if we  forget  the technical
requirements, symmetric  partitions may be  seen as 2-joins  such that
$X_i \setminus (A_i \cup B_i) = \emptyset$, $i=1,2$.

A lot of work has been done recently on finding algorithms that decide
if the  vertices of  a graph can  be partitioned into  several subsets
with           various          restrictions           on          the
adjacencies~\cite{cameron.e.h.s:listpartition,dantas.f.g.k:partition,feder.h.k.m:partition}).
Symmetric partitions  are detectable in linear  time (see problem~(31)
in~\cite{dantas.f.g.k:partition}).    Skew    partitions   seem   more
difficult,  but   Figueiredo,  Klein,  Kohayakawa  and   Reed  gave  a
polynomial time  algorithm that decides whether  a graph has  or not a
skew partition~\cite{figuereido.k.k.r:sp}. Note that detecting a $C_5$
in a graph  can be done easily in $O(n^5)$. Thus,  each of the outcome
of Theorem~\ref{th:m}  are testable in polynomial  time. We conjecture
that Theorem~\ref{th:m} holds for every sc-graph on $4k$ vertices:

\begin{conjecture}
  \label{conj}
  Let $G$ be an sc-graph on $4k$ vertices. Then either:

  \begin{itemize}
  \item
    $G$ contains a $C_5$ as an induced subgraph;
  \item
    $G$ contains a skew partition;
  \item 
    $G$ contains a symmetric partition.
  \end{itemize}
\end{conjecture} 

This conjecture is motivated by several considerations.  First, we are
able to prove it in  a quite general special case: Theorem~\ref{th:m}.
The proof shows  how forbiding $C_5$'s can help  for finding symmetric
or skew  partitions.  Also,  skew partitions and  symmetric partitions
arise naturally  in $P_4$-constructions  of sc-graphs and  in circular
permutations of an antimorphism.  Suppose  $(a_1 b_1 c_1 d_1 \dots a_k
b_k c_k d_k)$ is such a permutation with our usual notation.  If every
vertex  in  $A$  sees every  vertex  in  $B$,  then $(A,B,C,D)$  is  a
symmetric partition of $G[A\cup B \cup  C \cup D]$. If every vertex in
$A$ sees every vertex in $C$, then  $(B, D, C, A)$ is a skew partition
of $G[A\cup B \cup C \cup D]$.

Secondly,  the   conjecture  has  an  analogy  with   the  theorem  of
Chudnovsky,  Robertson,  Seymour  and  Thomas  for  decomposing  Berge
Graphs.  A  \emph{hole} in a  graph is an  induced cycle of  length at
least~4. A graph  is \emph{Berge} if in both  $G, \overline{G}$, there
is no hole  of odd length. The \emph{decomposition  theorem} for Berge
graphs is the following. Note that this theorem has been proved in two
steps:      first     Chudnovsky,      Robertson,      Seymour     and
Thomas~\cite{chudvovsky.r.s.t:spgt}  proved a slightly  weaker result,
and then Chudnovsky~\cite{chudnovsky:trigraphs}  alone proved the form
that we give:

\begin{theorem}[Chudnovsky et al.\cite{chudnovsky:trigraphs,chudvovsky.r.s.t:spgt}]
  \label{th.berge}
  Let $G$ be a Berge graph. Then either~:
  \begin{itemize}
    \item 
      One of $G, \overline{G}$ is bipartite.
    \item
      One of $G, \overline{G}$ is the line-graph of a bipartite graph.
    \item
      One of $G, \overline{G}$ has a 2-join.
    \item
      $G$ has a skew partition.
  \end{itemize}
\end{theorem}

It would be nice to have  a stronger theorem in the particular case of
Berge sc-graphs. Conjecture~\ref{conj} could be a candidate.  

\section*{Aknowledgements}

Many thanks to Fr\'ed\'eric Maffray for his encouragement and carefull
reading, and to Alastair Farrugia for numerous suggestions.


\providecommand{\bysame}{\leavevmode\hbox to3em{\hrulefill}\thinspace}
\providecommand{\MR}{\relax\ifhmode\unskip\space\fi MR }
\providecommand{\MRhref}[2]{%
  \href{http://www.ams.org/mathscinet-getitem?mr=#1}{#2}
}
\providecommand{\href}[2]{#2}

\begin{figure}[p]
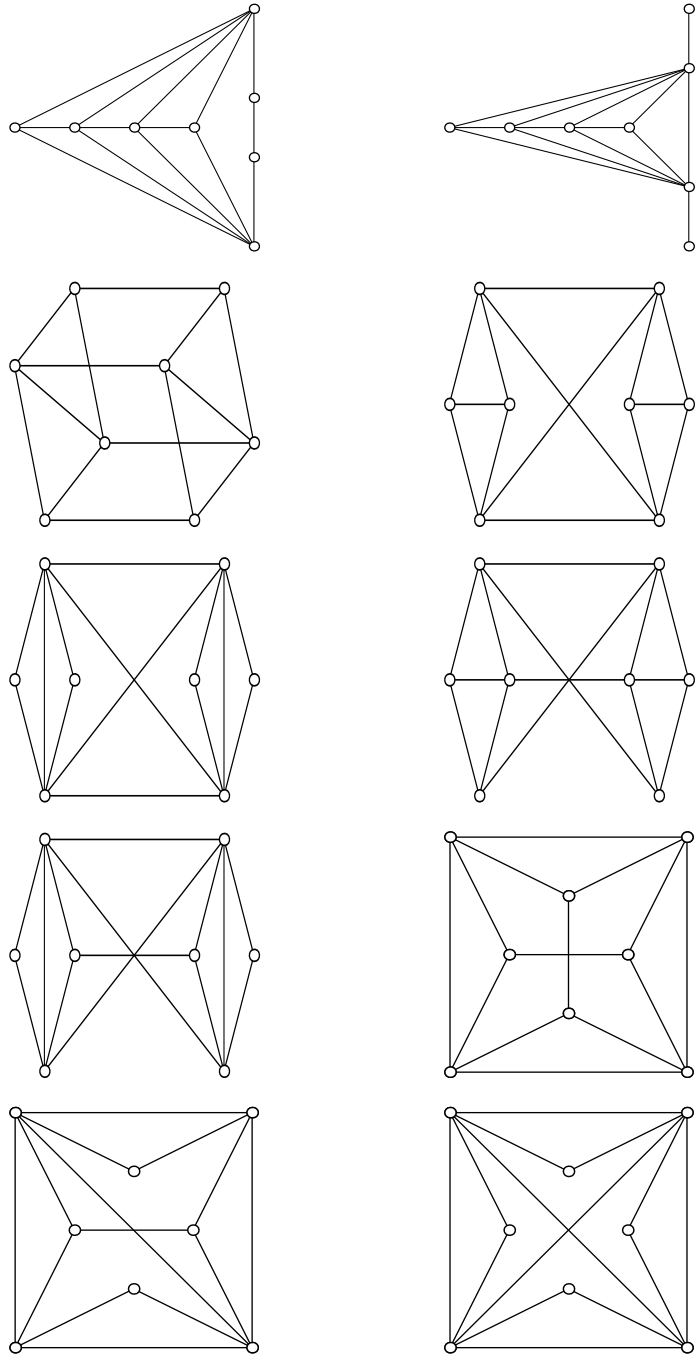

  \caption{The 10 sc-graphs on 8 vertices\label{fig8}}
 \center
\begin{tabular}{cc}
\includegraphics[width=3.3cm,height=3.3cm]{fig.autocomp.1}\rule{2cm}{0cm}\rule{0cm}{3.5cm}
&\includegraphics[width=3.3cm,height=3.3cm]{fig.autocomp.2}\\
\includegraphics[width=3.3cm,height=3.3cm]{fig.autocomp.3}\rule{2cm}{0cm}\rule{0cm}{3.5cm}
&\includegraphics[width=3.3cm,height=3.3cm]{fig.autocomp.4}\\
\includegraphics[width=3.3cm,height=3.3cm]{fig.autocomp.5}\rule{2cm}{0cm}\rule{0cm}{3.5cm}
&\includegraphics[width=3.3cm,height=3.3cm]{fig.autocomp.6}\\
\includegraphics[width=3.3cm,height=3.3cm]{fig.autocomp.7}\rule{2cm}{0cm}\rule{0cm}{3.5cm}
&\includegraphics[width=3.3cm,height=3.3cm]{fig.autocomp.8}\\
\includegraphics[width=3.3cm,height=3.3cm]{fig.autocomp.9}\rule{2cm}{0cm}\rule{0cm}{3.5cm}
&\includegraphics[width=3.3cm,height=3.3cm]{fig.autocomp.10}
\end{tabular}
\end{figure}

\end{document}